\documentclass[a4paper, 14pt]{amsart}

\usepackage[margin=1.15in]{geometry}
\usepackage{amscd,amssymb, amsmath, wasysym, mathrsfs, mathtools}
\usepackage{graphicx}
\usepackage[all, cmtip]{xy}

\usepackage{url}
\usepackage{hyperref}

\theoremstyle{plain}
\newtheorem{theorem}{Theorem}[section]
\newtheorem{prop}[theorem]{Proposition}

\newtheorem*{statement}{Statement}
\newtheorem{cor}[theorem]{Corollary}

\newtheorem{lemma}[theorem]{Lemma}

\newtheorem*{metatheorem}{Meta Theorem}
\theoremstyle{definition}

\newtheorem{defn}[theorem]{Definition}
\newtheorem{rmk}[theorem]{Remark}

\newtheorem{ex}[theorem]{Example}
\newtheorem*{ex*}{Example}

\newcommand\sA{{\mathcal A}}
\newcommand\sO{{\mathcal O}}
\newcommand\sK{{\mathcal K}}

\newcommand\sC{{\mathscr C}}

\newcommand\sE{{\mathcal E}}

\newcommand\sM{{\mathcal M}}
\newcommand\sW{{\mathcal W}}
\newcommand\sV{{\mathcal V}}
\newcommand\sL{\mathcal{L}}
\newcommand\sR{\mathcal{R}}

\newcommand\sS{\mathcal{S}}

\newcommand\zz{{\mathbb{Z}}}

\newcommand\cc{{\mathbb{C}}}

\newcommand\nn{{\mathbb{N}}}

\newcommand\hh{{\mathbb{H}}}
\newcommand\br{{\mathbf{R}}}
\newcommand\sEnd{\mathcal{E}nd^\bullet}
\newcommand\sEndp{\mathcal{E}nd^{0, \bullet}}
\newcommand{\odol}{\Omega^{p, \bullet}_{\textrm{Dol}}}
\newcommand{\mb}{\mathcal{M}_{\textrm{B}}}
\newcommand{\mdr}{\mathcal{M}_{\textrm{DR}}}
\newcommand{\mdol}{\mathcal{M_{\textrm{Dol}}}}

\DeclareMathOperator{\id}{id}                    
\DeclareMathOperator{\obj}{Obj}
\DeclareMathOperator{\ad}{ad}
\DeclareMathOperator{\morph}{Morph}
\DeclareMathOperator{\enom}{End}
\DeclareMathOperator{\iso}{Iso}
\DeclareMathOperator{\homo}{Hom}
\DeclareMathOperator{\enmo}{End}
\DeclareMathOperator{\spec}{Spec}

\DeclareMathOperator{\odr}{\Omega^\bullet_{\textrm{DR}}}

\DeclareMathOperator{\red}{red}

\DeclareMathOperator{\kn}{Ker}
\DeclareMathOperator{\im}{Im}

\title{Cohomology jump loci of compact K\"ahler manifolds}
\begin{document}
\author{Botong Wang}
\date{}
\begin{abstract}
We apply the method of Dimca-Papadima to study the cohomology jump loci in the representation variety and the moduli space of vector bundles with vanishing chern classes for a compact K\"ahler manifold. We introduce modules over differential graded Lie algebra to extend results of Dimca-Papadima to a general point in the representation variety or the moduli space. We show that locally the cohomology jump loci is isomorphic to the resonance variety via the exponential map. This paper generalizes a previous result of the author. 
\end{abstract}
\maketitle
\section{Introduction}
Given a connected CW-complex $X$ of finite type with a base point $x\in X$, the set of group homomorphisms $\homo(\pi_1(X, x), Gl(n, \mathbb{C}))$ has naturally a scheme structure. We denote this scheme by $\mathbf{R}(X, n)$. Every closed point $\rho\in\mathbf{R}(X, n)$ corresponds to a rank $n$ local system $L_\rho$ on $X$, with a basis $L_\rho|_{x}\cong \mathbb{C}^n$. Conversely, every rank $n$ local system with a basis at $x$ gives rise to a $Gl(n, \mathbb{C})$ representation of $\pi_1(X, x)$. In $\mathbf{R}(X, n)$, there are some canonically defined loci, $\sV^i_r(X, n)\stackrel{\textrm{def}}{=}\{\rho\in \mathbf{R}(X, n)\;|\; \dim H^i(X, L_\rho)\geq r\}$. They are called the cohomology jump loci, and they were studied extensively when $X$ is a smooth (quasi-)projective variety and $r=1$. For example, when $X$ is a projective variety, in a sequence of works by Green-Lazarsfeld \cite{gl1}, \cite{gl2}, Arapura, \cite{a2} and Simpson \cite{s2}, it is proved that $\sV^i_r(X, 1)\subset \mathbf{R}(X, 1)$ are torsion translates of subtori. More recently, in \cite{bw}, this result is generalized to all quasi-projective varieties. 

In \cite{dp}, Dimca and Papadima studied the analytic germ of $\mathbf{R}(X, n)$ at the trivial representation, and related it to the rational homotopy type of $X$. The purpose of this paper is to generalize this kind of results to points away from the trivial representation. To do that, we need to replace a commutative differential graded algebra (CDGA) by a differential graded Lie algebra (DGLA) as in \cite{gm}, and introduce modules over DGLA. More precisely, we will be considering a DGLA pair consisting of a DGLA and a module over that DGLA. The DGLA controls the deformation theory of the representation variety, and the module controls the deformation of the cohomology groups. Since we do not have a substitute for Morgan's Gysin complex for a nontrivial representation, we will mostly consider compact K\"ahler manifolds in this paper. 

Let $X$ be a connected CW complex. The following principle is demonstrated in \cite{dp}.

\begin{metatheorem}
The analytic germs of the representation variety $\mathbf{R}(X, n)$ and the cohomology jump loci $\mathcal{V}^i_r(X, n)$ at the trivial representation are controlled by the rational homotopy type of $X$. When $X$ is a smooth manifold, they are controlled by the homotopy type of the de Rham complex of $X$. 
\end{metatheorem}

More precisely, for a CDGA $(\mathcal{A}^\bullet, d)$, they defined a space of flat connections 
$$
\mathcal{F}(\sA^\bullet, n)=\{\omega\in \mathcal{A}^1\otimes \mathfrak{gl}(n, \mathbb{C})\; |\; d\omega+\frac{1}{2}[\omega, \omega]=0\}
$$
and the relative resonance variety 
$$
\mathcal{R}^i_r(\mathcal{A}^\bullet, n)=\{\omega \in \mathcal{F}(\mathcal{A}^\bullet, n)\;|\; \dim H^i(\mathcal{A}^\bullet\otimes \mathbb{C}^n, d\otimes \id_{\cc^n}+\omega)\geq r\}. 
$$
Suppose $X$ is a smooth manifold. Denote the de Rham complex of $X$ by $\odr(X)$. Assume $\sA^\bullet$ is of the same homotopy type as $\odr(X)$, and assume $\dim_\mathbb{C} \mathcal{A}^0=1$, $\dim_\mathbb{C} \mathcal{A}^{\leq q} <\infty$ for every $q\in \zz$. Then they showed that after taking the reduced formal schemes, there is a canonical isomorphism
$$
\mathbf{R}(X, n)_{(1), \red}\cong \mathcal{F}(\mathcal{A}^\bullet, n)_{(0), \red}
$$
which induces an isomorphism
$$
\mathcal{V}^i_r(X, n)_{(1), \red}\cong \mathcal{R}^i_r(\mathcal{A}^\bullet, n)_{(0), \red}
$$
\begin{rmk}
The results of \cite{gm} and \cite{dp} are about the analytic germs. And the notation $\mathbf{R}(X, n)_{(1), \red}$ was used there for analytic germs. However, according to \cite{ar}, the isomorphism between two analytic germs is equivalent to the isomorphism between the formal schemes. In this paper, we use the subscript ($\cdot$) for taking formal scheme at a point. Thus, essentially this is same as taking analytic germs. 
\end{rmk}

Let $\rho\in Hom(\pi_1(X), GL(n, \mathbb{C}))$ be any representation. We denote its induced local system by $L_\rho$, and define $\odr(L_\rho)=\odr(X)\otimes_{\mathbb{C}}L_\rho$, $\mathcal End^\bullet(L_\rho)=\Omega^\bullet_{DR}(X)\otimes_\mathbb{C}End(L_\rho)$, where $End(L_\rho)=L_\rho\otimes_\mathbb{C} L_\rho^*$ is the local system of endomorphisms of $L_\rho$. Fixing the base point $x$ in $X$, there is an augmentation map $\epsilon_x: \mathcal End^\bullet(L_\rho)\to End(L_\rho|_x)$, which is the restriction map on degree zero, and is zero in other degrees. Denote the kernel of $\epsilon_x$ by $\mathcal End^\bullet_0(L_\rho)$. Then our new principle is the following. 

\begin{metatheorem}
The analytic germs of the representation variety $\mathbf{R}(X, n)$ and the cohomology jump loci $\mathcal{V}^i_r(X, n)$ at $\rho$ are controlled by the homotopy type of the pair $(\sEnd_0(L_\rho), \Omega^\bullet_{DR}(L_\rho))$.
\end{metatheorem}

When $X$ is a connected compact K\"ahler manifold and when $\rho$ is a semi-simple representation, both $\mathcal End^\bullet(L_\rho)$ and $\Omega^\bullet_{DR}(L_\rho)$ are formal (see \cite{s1}). In this case, we can use the above principle to obtain some properties of $\mathcal{V}^i_r(X, n)$.

It is well known that the Zariski tangent space of $\mathbf{R}(X, n)$ at $\rho$ is isomorphic to the vector space $Z^1(\pi_1(X), \mathfrak{gl}(n, \mathbb{C})_{\textrm{Ad}\rho})$. In \cite{gm}, Goldman and Millson showed that locally, $\mathbf{R}(X, n)$ is isomorphic to a quadratic cone $C$ in $Z^1(\pi_1(X), \mathfrak{gl}(n, \mathbb{C})_{\textrm{Ad}\rho})$. Since $H^1(X, End(L_\rho))$ computes the group cohomology $H^1(\pi_1(X), \mathfrak{gl}(n, \mathbb{C})_{\textrm{Ad}\rho})$, we have naturally
$$
H^1(X, End(L_\rho))\cong Z^1(\pi_1(X), \mathfrak{gl}(n, \mathbb{C})_{\textrm{Ad}\rho})/B^1(\pi_1(X), \mathfrak{gl}(n, \mathbb{C})_{\textrm{Ad}\rho}).
$$
For a cocycle $\eta\in Z^1(\pi_1(X), \mathfrak{gl}(n, \mathbb{C})_{\textrm{Ad}\rho})$, denote its image in $H^1(X, End(L_\rho))$ under the above isomorphism by $\bar\eta$. In fact, the cone $C$ is defined to be the kernel of $Z^1(\pi_1(X), \mathfrak{gl}(n, \mathbb{C})_{\textrm{Ad}\rho})\to H^2(X, End(L_\rho))$ sending $\eta$ to $\bar\eta\wedge \bar\eta$. When $\eta\in C$, the following is the Aomoto complex of $\eta$, 
$$
(H^\bullet(X, L_\rho), \wedge\bar\eta)\stackrel{\textrm{def}}{=}H^0(X, L_\rho)\stackrel{\wedge \bar\eta}{\longrightarrow}H^1(X, L_\rho)\stackrel{\wedge \bar\eta}{\longrightarrow}H^2(X, L_\rho)\stackrel{\wedge \bar\eta}{\longrightarrow}\cdots.
$$
We set-theoretically define the resonance variety $\mathcal{R}^i_r(X, \rho)=\{\eta\in C\;|\; \dim H^i(H^\bullet(X, L_\rho), \bar\eta)\geq r\}$, which is a cone in $C$. A precise definition of $\mathcal{R}^i_r(X, \rho)$ as closed subscheme of $C$ will be given in Definition \ref{resonance}. 

\begin{theorem}\label{thm1}
As showed in \cite{gm}, when $\rho$ is semisimpe, there is an isomorphism\footnotemark\  between the formal schemes
$$\mathbf{R}(X, n)_{(\rho)}\cong C_{(0)}.$$
This isomorphism induces an isomorphism between the reduced formal schemes
$$\mathcal{V}^i_r(X, n)_{(\rho), \red}\cong \mathcal{R}^i_r(X, \rho)_{(0), \red}.$$
\end{theorem}
\footnotetext{Even though the isomorphism may not be canonically defined, we will see that the ambiguity is an isomorphism of $C_{(0)}$ which preserves each $\sR^i_r(X, \rho)$. }

The non-reduced part of $\mathcal{V}^i_r(X, n)$ near a point $\rho\in \mathbf{R}(X, n)$ with $\dim H^i(X, L_\rho) > r$ is mysterious. However, as discussed in \cite{w}, the deformation theory of $\mathcal{V}^i_r(X, n)$ at $\rho$ with $\dim H^i(X, L_\rho)=r$ is well understood. In particular, we can generalize the previous theorem to the original formal scheme before taking reduced induced structure.

\begin{theorem}\label{thm2}
As in the previous theorem, moreover we assume $\dim H^i(X, L_\rho)=r$. Then the isomorphism between $\mathbf{R}(X, n)_{(\rho)}$ and $C_{(0)}$ induces an isomorphism
$$\mathcal{V}^i_r(X, n)_{(\rho)}\cong \mathcal{R}^i_r(X, \rho)_{(0)}.$$
\end{theorem}
It is very easy to see that when $\dim H^i(X, L_\rho)=r$, $\sR^i_r(X, \rho)$ is the intersection of $C$ with a linear subspace of $Z^1(\pi_1(X), \mathfrak{gl}(n, \cc)_{\ad \rho})$. Therefore, a consequence of Theorem \ref{thm2} is the following. 
\begin{cor}\label{consequence}
Let $\rho\in \br(X, n)$ be a semi-simple representation, and let $r=\dim H^i(X, L_\rho)$. Then $\sV^i_r(X, n)$ has quadratic singularities at $\rho$. 
\end{cor}
The main part of the paper is to work out Sections 8, 9 and 10 of \cite{dp} in this new setting. The ideas here are very similar to \cite{dp}, but our presentation is closer to \cite{gm}. We will give complete arguments to the main theorems without quoting the arguments of \cite{dp}, except some definitions and immediate properties of the definitions. 

In section 4, we will use DGLA pairs to study the cohomology jump loci in the moduli space of stable holomorphic vector bundles of rank $n$. In this case, since the tangent space is represented by forms of pure type (0, 1), it allows us to work with $H^{p, q}$ cohomology jump loci. We will show that they are also controlled by a DGLA pair which is formal. Then we can prove the parallel versions of Theorem \ref{thm1} and Theorem \ref{thm2}, which will appear as Theorem \ref{thm3}. 

In the last section, we will mention some application to conditions on fundamental groups of K\"ahler manifolds and some potential connections with absolute constructible sets defined by Simpson. 

Some conventions we use in this paper are the following. We will compress the base point $x\in X$, and only write $\pi_1(X)$ for $\pi_1(X, x)$. The subscript ($\cdot$) means either the formal scheme at a point or the completion of the localization of the structure sheaf  at some point. For example, $\br(X, n)_{(\rho)}$ is the formal scheme of $\br(X, n)$ at point $\rho$, and $(\sO_{\br(X, n)})_{(\rho)}$ is the completion of the local ring $\sO_{\br(X, n), \rho}$. We will frequently use the same notation for a point in a moduli space and the object which that point corresponds to. All the rings and vector spaces are over $\cc$, and $\dim$ is always taking the complex dimension unless we specify with a subscript. When we say an Artinian local ring, we always assume the ring to be finitely generated over $\cc$.

\section{DGLA pairs and the cohomology functor}
We recall the definition of a differential graded Lie algebra (DGLA) over $\mathbb{C}$, as defined in \cite{gm}. 

\begin{defn}
A DGLA consists of the following set of data,
\begin{enumerate}
\item a graded vector space $C=\bigoplus_{i\in \mathbb{Z}}C^i$ over $\mathbb{C}$,
\item a Lie bracket which is bilinear, graded commutative and satisfies the graded Jacobi identity, i.e., for any $\alpha\in C^i, \beta\in C^j$ and $\gamma\in C^k$,
$$[\alpha, \beta]=(-1)^{ij}[\beta, \alpha]$$
and
$$(-1)^{ki}[\alpha, [\beta, \gamma]]+(-1)^{ij}[\beta, [\gamma, \alpha]]+(-1)^{jk}[\gamma, [\alpha, \beta]]=0$$
\item a family of linear maps, called the differential maps, $d^i: C^i\to C^{i+1}$, satisfying $d^{i+1}d^i=0$ and the Leibniz rule, i.e., for $\alpha\in C^i$ and $\beta\in C$
$$d[\alpha, \beta]=[d\alpha, \beta]+(-1)^i[\alpha, d\beta]$$
where $d=\sum d^i: C\to C$. 
\end{enumerate}
A homomorphism of DGLAs is a linear map which preserves the grading, Lie bracket, and the differential maps. 
\end{defn}
We denote this DGLA by $(C, d)$, or $C$ when there is no risk of confusion. 

\begin{defn}\label{module}
Given a DGLA $(C, d_C)$, we define a module over $(C, d_C)$ to be the following set of data,
\begin{enumerate}
\item a graded vector space $M=\bigoplus_{i\in \mathbb{Z}} M^i$ together with a bilinear multiplication map $C\times M\to M$, $(a, \xi)\mapsto a\xi$, such that for any $\alpha\in C^i$ and $\xi\in M^j$, $\alpha\xi \in M^{i+j}$. And furthermore, for any $\alpha\in C^i, \beta\in C^j$ and $\zeta\in M$, we require
$$[\alpha, \beta]\zeta=\alpha(\beta\zeta)-(-1)^{ij}\beta(\alpha\zeta).$$
\item a family of linear maps $d^i_M: M^i\to M^{i+1}$ (write $d_M=\sum_{i\in\zz} d^i_M: M\to M$), satisfying $d^{i+1}_M d^i_M=0$. And we require it to be compatible with the differential on $C$, i.e., for any $\alpha\in C^i$, 
$$d_M(\alpha\xi)=(d_C\alpha)\xi+(-1)^i\alpha(d_M\xi).$$
\end{enumerate}
\end{defn}

We will call such a module by a $(C, d_C)$-module or simply a $C$-module. 

\begin{defn}
A homomorphism of $(C, d_C)$-modules $f: (M, d_M)\to (N, d_N)$ is a linear map $f: M\to N$ which satisfies
\begin{enumerate}
\item $f$ preserves the grading, i.e., $f(M^i)\subset N^i$,
\item $f$ is compatible with multiplication by elements in $C$, i.e., $f(\alpha\xi)=\alpha f(\xi)$, for any $\alpha\in C$ and $\xi \in M$,
\item $f$ is compatible with the differentials, i.e., $f(d_M\alpha)=d_N f(\alpha)$.
\end{enumerate}
\end{defn}
Fixing a DGLA $(C, d_C)$, the category of $C$-modules is an abelian category. 

\begin{defn}
A DGLA pair is a DGLA $(C, d_C)$ together with a $(C, d_C)$-module $(M, d_M)$. Usually, we write such a pair simply by $(C, M)$. A homomorphism of DGLA pairs $g: (C, M)\to (D, N)$ consists of a map $g_1: C\to D$ of DGLA and a $C$-module homomorphism $g_2: M\to N$, considering $N$ as a $C$-module induced by $g_1$. We call $g$ a homotopy equivalence if $g_1$ and $g_2$ both induce isomorphisms on the cohomology groups. And we define two DGLA pairs to be of the same homotopy type, if they can be connected by a zigzag of homotopy equivalences. 
\end{defn}
Let $(C, M)$ be a DGLA pair. Then $H^\bullet(C)$, the cohomology of $C$ with zero differentials, is a DGLA. $H^\bullet(M)$, the cohomology of $M$ with zero differentials, is an $H^\bullet(C)$-module. 
\begin{defn}
We call the DGLA pair $(H^\bullet(C), H^\bullet(M))$ the \textbf{cohomology DGLA pair} of $(C, M)$. Moreover, we say DGLA pair $(C, M)$ is \textbf{formal}, if $(C, M)$ is of the same homotopy type as $(H^\bullet(C), H^\bullet(M))$. 
\end{defn}

Given a DGLA $(C, d)$ over $\mathbb{C}$ together with an Artinian local ring $A$ of finite type over $\mathbb{C}$, a groupoid $\mathcal{C}(C, A)$ is defined in \cite{gm}, which in the language of \cite{dp} is the space of flat connections. We recall their definition.

$C\otimes_{\cc}A$ is naturally a DGLA by letting $[\alpha\otimes a, \beta\otimes b]=[\alpha, \beta]\otimes ab$ and $d(\alpha\otimes a)=d\alpha \otimes a$. Let $m$ be the maximal ideal in $A$. Then under the same formula, $C\otimes_{\cc} m$ is also a DGLA. Since $(C\otimes_{\cc}m)^0=C^0\otimes_{\cc}m$ is a nilpotent Lie algebra, the Campbell-Hausdorff multiplication defines a nilpotent Lie group structure on the space $C^0\otimes m$. We denote this Lie group by $\exp(C^0\otimes m)$. Now, Lie group $\exp(C^0\otimes m)$ acts on $C^1$ by
$$\overline\exp(\lambda): \alpha \mapsto \exp(\ad \lambda)\alpha+\frac{1-\exp(\ad \lambda)}{\ad \lambda}(d\lambda)$$
in terms of power series. 

\begin{defn}\label{cat}
Category $\sC(C; A)$ is defined to be the category with objects
$$\obj \sC(C; A)=\{\omega\in C^1\otimes_{\cc}m\;|\; d\omega+\frac{1}{2}[\omega, \omega]=0\},$$
and with the morphisms between two elements $\omega_1$, $\omega_2$
$$\morph(\omega_1, \omega_2)=\{\lambda\in \exp(C^0\otimes m)\;|\; \exp(\lambda)\omega_1=\omega_2\}.$$
\end{defn}
Given any $\omega\in \obj\sC(C; A)$, it gives rise to a new DGLA, $(C\otimes A, d_C\otimes \id_A+\ad \omega)$. The following equivalence theorem was observed by Deligne \cite{gm} and Schlessinger-Stasheff \cite{ss}. 
\begin{theorem}\label{deligne}
Let $\psi: C\to D$ be a homomorphism of DGLAs. Suppose the induced map on cohomology $H^i(\psi): H^i(C)\to H^i(D)$ are isomorphisms for $i=0, 1$ and injective for $i=2$. Let $A$ be any Artinian local ring of finite type over $\cc$. Then the induced functor
$$\psi_*: \sC(C; A)\to \sC(D; A)$$
is an equivalence of groupoids. In particular, if $C$ and $D$ have the same homotopy type, then the set of isomorphism classes $\iso \sC(C; A)=\iso \sC(D; A)$. 
\end{theorem}

\begin{ex}
One important example we should keep in mind is the following (see \cite{gm} 9.4). Let $X$ be a compact K\"ahler manifold, and let $E$ be a holomorphic vector bundle on $X$. Then $\enom(E)\cong E\otimes E^{\vee}$ is also a holomorphic vector bundle on $X$. Let $(\mathcal End(E)^{0, \bullet}, \bar\partial)$ be the Dolbeault complex with values in $\enom(L)$. For any artinian local ring $A$, the isomorphism classes of the groupoid $\sC(\mathcal End(E)^{0. \bullet}; A)$ correspond to the isomorphism classes of extensions $E_A$ of $E$, that is a locally free $\sO_X\otimes_{\cc} A$ module $E_A$ on $X$ such that $E_A\otimes_A \cc\cong E$. 
\end{ex}

\begin{defn}\label{cohfun}
Given a DGLA pair $(C, M)$, we define a functor $\hh^i_M: \sC(C; A)\to \mathfrak{Mod}(A)$ for each $i\in \zz$, where $\mathfrak{Mod}(A)$ is the category of $A$-modules. In fact, $(M\otimes_{\cc}A, d_M\otimes \id_A)$ is naturally a $(C\otimes A, d_C\otimes \id_A)$-module. It follows from property (1) of Definition \ref{module}, that for any $\omega\in \obj(\sC(C; A))$, $d_M\otimes \id_A+\omega$ is a differential on $M\otimes A$, i.e., $(d_M\otimes \id_A+\omega)^2=0$. Therefore, $(M\otimes A, d_M\otimes \id_A+\omega)$ is a $(C\otimes A, d_C\otimes \id_A+\ad \omega)$-module. We define $\hh^i_M(\omega)=H^i(M\otimes A, d_M\otimes \id_A+\omega)$. Since $(M\otimes A, d_M\otimes \id_A+\omega)$ is a complex of $A$-modules, $\hh^i_M(\omega)$ has naturally an $A$-module structure. Given any $\lambda\in C^0\otimes m$, $\exp(\lambda)$ acts on $M\otimes A$ by the usual power series action 
$$\exp(\lambda)(\xi)=\xi+\lambda\xi+\frac{1}{2}\lambda(\lambda\xi)+\cdots$$
for any $\xi\in M\otimes A$. And since $C^0\otimes m$ is nilpotent, the sum above is indeed a finite sum. 

A direct computation shows that for any $\lambda\in C^0\otimes m$, the commutativity of the following diagram can be deduced from the next lemma. 
\[
\xymatrixcolsep{8pc}\xymatrix{
M^i\otimes A\ar[r]^{d_M\otimes \id_A+\omega}\ar[d]^{\exp(\lambda)}& M^{i+1}\otimes A\ar[d]^{\exp(\lambda)}\\
M^i\otimes A\ar[r]^{d_M\otimes \id_A+\overline\exp(\lambda)(\omega)}&M^{i+1}\otimes A
}
\]
Hence, $\exp(\lambda): M\to M$ induces a map of complexes of $A$-modules,
\begin{equation}\label{exlambda}
\exp(\lambda): (M\otimes A, d_M\otimes \id_A+\omega)\to (M\otimes A, d_M\otimes \id_A+\exp(\lambda)(\omega))
\end{equation}
We define a functor $\hh^i_M: \sC(C; A)\to \mathfrak{Mod}(A)$ mapping an object $\omega\in C^1\otimes_{\cc} m$ to $H^i(M\otimes_\cc A, d_M\otimes \id_A+\omega)$, and mapping a morphism $\exp(\alpha)$ to the induced map on cohomology. The above argument shows that $\hh^i_M$ is a well defined functor. By putting $\hh_M=\bigoplus_{i\in \zz}\hh^i_M$, we obtain a functor $\hh_M: \sC(C; A)\to \mathfrak{Mod}_\zz(A)$, where $\mathfrak{Mod}_\zz(A)$ is the category of $\zz$ graded $A$-modules. 
\end{defn}

\begin{lemma}
Under the above notations, for any $\lambda\in C^0\otimes m$, $\omega\in \obj(\sC(C; A))$, and $\xi\in M\otimes A$, the following equations hold.

\begin{equation}
\exp(\lambda)(\omega\xi)=(\exp(\ad \lambda)\omega)\exp(\lambda)\xi
\end{equation}
\begin{equation}
\exp(\lambda)d\xi=d(\exp(\lambda)\xi)+\left(\frac{1-\exp(\ad \lambda)}{\ad \lambda}d\lambda\right)\exp(\lambda)\xi
\end{equation}
\end{lemma}

\begin{proof}
For (1), we expend the right side of the equation, and calculate the coefficient of term $\lambda^p \omega \lambda^q \xi$. It is equal to 
\begin{align*}
\sum_{i=0}^q(-1)^{q-i}\,\frac{1}{i!}\frac{1}{(p+q-i)!}{p+q-i\choose p}&=\sum_{i=0}^q(-1)^{q-i}\,\frac{1}{i!p!(q-i)!}\\
&=(-1)^q\frac{1}{p!q!}\sum_{i=0}^q\left((-1)^{i}{p \choose i}\right)
\end{align*}
The last sum is zero unless $q=0$, and in this case, the coefficient is $\frac{1}{p!}$. This is exactly the coefficient of $\lambda^p\omega\xi$ on the left side of the equation. 

To show (2), by comparing the coefficient of the term $\lambda^p(d\lambda)\lambda^q\xi$, we are lead to show the following equality,
$$\frac{1}{(p+q+1)!}=\sum_{i=0}^{q}\frac{(-1)^{p-i}}{i! (p+q-i+1)!} {p+q-i\choose p}$$ 
and this is equivalent to
$$\frac{p!q!}{(p+q+1)!}=\sum_{i=0}^q\frac{(-1)^{q-i}}{p+1+q-i}{q\choose i}.$$
Now, the right side is equal to $\int_0^1(1-t)^qt^pdt$. And by induction on $q$, we can easily show the integration is equal to $\frac{p!q!}{(p+q+1)!}$. 
\end{proof}

Let $g=(g_1, g_2): (C, M)\to (D, N)$ be a homomorphism of DGLA pairs, and let $g_{1*}: \sC(C; A)\to \sC(D; A)$ be the induced functor on the groupoids. For any $\omega\in \obj\sC(C; A)$, $g_2: M\to N$ induces a homomorphism of complexes, or more precisely, a homomorphism of $(C\otimes A, d_C\otimes \id_A+\ad\omega)$-modules, 
\begin{equation}\label{phi2}
g_2\otimes \id_A: (M\otimes A, d_M\otimes \id_A+\omega)\to (N\otimes A, d_N\otimes \id_A+g_1(\omega)).
\end{equation}
Here $(N\otimes A, d_N\otimes \id_A+g_1(\omega))$ has a $(C\otimes A, d_C\otimes \id_A+\ad\omega)$-module structure induced from the homomorphism of DGLA 
$$g_1\otimes \id_A: (C\otimes A, d_C\otimes \id_A+\ad\omega)\to (D\otimes A, d_D\otimes id_A+\ad (g_1(\omega))).$$
Taking cohomology of the homomorphism (\ref{phi2}), we obtain a homomorphism of $\zz$ graded $A$-modules,
$$\hh(g_2): \hh_M(\omega)\to \hh_N(g_{1*}(\omega)).$$

\begin{prop}\label{homequ}
Under the notations as the previous paragraph, if $g$ is a homotopy equivalence, then $\hh(g_2)$ is an isomorphism. 
\end{prop}
\begin{proof}
First, notice that if we consider $N$ as a $C$-module, then $\hh_N(g_{1*}(\omega))$ would be naturally isomorphic to $\hh_N(\omega)$. Therefore, without loss of generality, we can assume that $C=D$ and $g_1=\id_C$. 

Under the above assumption, we prove the proposition by induction on the dimension of $A$ as a $\cc$ vector space. When $\dim A=1$, that is $m=0$, the proposition is trivial. Suppose the proposition is true for Artinian local rings whose dimension as a $\cc$ vector space is less than $\dim A$. Now, in $A$ there is always an ideal which is isomorphic to $\cc$ as a $\cc$ vector space. For example, let $l$ be the largest integer such that $m^l$ is nonzero. Then any one dimensional subspace of $m^l$ would work. Hence, we have a short exact sequence
$$0\to \cc\to A\to A'\to 0$$
where $A'=A/\cc$, considering $\cc$ as an ideal of $A$. This is indeed a short exact sequence of $A$-modules. Take the tensor product of the complex $(M\otimes A, d_M\otimes \id_A+\omega)$ with $0\to \cc\to A\to A'\to 0$ over $A$. Since $M^i\otimes A$ are flat $A$-modules, we obtain a short exact sequence of complexes,
$$
0\to(M, d_M)\to (M\otimes A, d_M\otimes \id_A+\omega)\to (M\otimes A', d_M\otimes \id_A'+\omega')\to 0
$$
where $\omega'$ is the image of $\omega$ under the functor $\sC(C; A)\to \sC(C; A')$ induced by $A\to A'$. Thus we have a long exact sequence
\begin{equation}\label{long1}
\cdots\to H^i(M, d_M)\to \hh^i_M(\omega)\to \hh^i_M(\omega')\to \cdots
\end{equation}
Doing the same for $C$-module $N$, we also obtain a long exact sequence
\begin{equation}\label{long2}
\cdots\to H^i(N, d_N)\to \hh^i_N(\omega)\to \hh^i_N(\omega')\to \cdots
\end{equation}
Moreover, $g_2: M\to N$ induces a map $g_{2*}$ between the two long exact sequences (\ref{long1}) and (\ref{long2}). By the assumption that $g_2$ is a homotopy equivalence, the map $g_{2*}: H^i(M, d_M)\to H^i(N, d_N)$ is an isomorphism for every $i\in \zz$. And by induction hypothesis, $g_{2*}: \hh^i_M(\omega')\to \hh^i_N(\omega')$ is also an isomorphism for every $i\in \zz$. By the 5-lemma, $g_{2*}: \hh^i_M(\omega)\to \hh^i_N(\omega)$ is an isomorphism as well, and this is degree $i$ component of the map $\hh(g_2)$ in the proposition. 
\end{proof}
In fact, we don't need $g_1$ to be homotopy equivalence to prove the proposition. Removing it from the definition of a homomorphism of DGLA pairs to be homotopy equivalence will not affect the rest of the paper. 

\section{Cohomology jump loci in the representation varieties}
Let us first assume $X$ to be a smooth manifold with a base point $x\in X$. Let $\mathbf{R}(X, n)$ be the representation variety of $X$. The cohomology jump loci $\sV^i_k(X)$ in the representation variety $\mathbf{R}(X, n)$ are defined in \cite[9.1]{dp} as the subscheme associated to the fitting ideal of some universal cochain complex $P^\bullet$. Here $P^\bullet$ is a complex of locally free coherent sheaves on $\mathbf{R}(X, n)$. Instead of repeating their definition here, we give a characterization of $P^\bullet$. 

Over $X\times \br(X, n)$ there is a universal local system $\sL$. Denote the projections from $X\times \br(X, n)$ to its first and second factor by $p_1$ and $p_2$ respectively. Then $\sL$ is a $p_2^*(\sO_{\br(X, n)})$ module, satisfying the following property. For every closed point $\rho\in \mathbf{R}(X, n)$, the local system $L_{\rho}$ associated to the representation $\rho$ is isomorphic to $\sL|_{X\times\{\rho\}}$. The universal cochain complex $P^\bullet$ is a complex of free sheaves on $\br(X, n)$ satisfying for any $\sO_{\br(X, n)}$-module $G$: there is an isomorphism, functorial on $G$,
\begin{equation}\label{char}
\mathbf{R}p_{2*}^i(\sL\otimes_{p^*_2\sO_{\br(X, n)}} p^*_2 G)\cong H^i(P^\bullet \otimes_{\sO_{\br(X, n)}} G).
\end{equation}
In particular, let $S$ be any noetherian ring over $\cc$, and let $s: \spec(S)\to \br(X, n)$ be a homomorphism of schemes. Denote $X\times \spec(S)$ by $X_S$ and denote the pull-back of $\sL$ by $\id_X\times s: X\times \spec(S)\to X\times \br(X, n)$ by $L_S$. Regarding $S$ as an $\sO_{\br(X, n)}$-module via $s$, then (\ref{char}) implies
\begin{equation}\label{iso2}
H^i(X_S, L_S)\cong \Gamma(\br(X, n), H^i(P^\bullet \otimes_{\sO_{\br(X, n)}} S)).
\end{equation}

Let $d_{i-1}: P^{i-1}\to P^i$ and $d_i: P^i\to P^{i+1}$ be the differentials of $P^\bullet$, and denote the rank of $P^i$ by $l_i$. As in \cite{dp}, $\sV^i_k(X, n)$ is defined by the determinantal ideal $I_{l_i-k+1}(d_{i-1}\oplus d_i)$. Then the basic property of determinantal ideals and the isomorphism (\ref{iso2}) implies that the closed points in $\sV^i_k(X, n)$ are exactly $\{\rho\in \br(X, n)\;|\;\dim H^i(X, L_\rho)\geq k\}$. 

Fixing a closed point $\rho\in \br(X, n)$, we denote the associated local system of $\rho$ by $L_\rho$, and denote the $C^\infty$ de Rham complex of $L_\rho$ by $(\odr(L_\rho), \nabla)$. Let $\enmo(L_\rho)$ be the local system of local sections of endomorphisms of $L_\rho$. Similarly, denote the $C^\infty$ de Rham complex of $\enmo(L_\rho)$ by $(\sEnd(L_\rho), \nabla)$. Then $(\sEnd(L_\rho), \odr(L_\rho))$ is a DGLA pair. Restricting to $x\in X$ defines a homomorphism of Lie algebras $\mathcal End^0(L_\rho)\to \enmo(L_\rho)|_x$. If we consider $\enmo(L_\rho)|_x$ as a DGLA with only nonzero elements in degree zero, then we have naturally a DGLA homomorphism $\varepsilon_x: \sEnd(L_\rho)\to \enmo(L_\rho)|_x$. Denote the kernel of $\varepsilon _x$ by $\sEnd_0(L_\rho)$.
According to \cite[6.8]{gm}, the deformation theory of $\br(X, n)$ at $\rho$ is controlled by the DGLA $(\sEnd_0(L_\rho), \nabla)$. 

\begin{theorem}[Goldman-Millson]\label{gm1}
Denote the complete ring $(\sO_{\br(X, n)})_{(\rho)}$ by $R$. The functor $A\mapsto \iso \sC(\sEnd_0(L_\rho); A)$ from the category of Artinian local rings to the category of sets is represented by $R$. 
\end{theorem}
\begin{proof}[Sketch of proof]
Since any ring homomorphism from $\sO_{\br(X, n), \rho}$ to $A$ factors through $R$, the subset of $\homo(\spec(A), \br(X, n))$ whose theoretic image is $\rho$ is functorially bijective to $\homo(R, A)$. And $\homo(\spec(A), \br(X, n))$ is functorially bijective to the isomorphism classes of $A$-local systems with a base, i.e., $(L_A, \mathfrak{b})$, where $L_A$ is a locally constant free $A$ module of rank $n$ on $X_A=X\times \spec(A)$ whose restriction to $X$ is isomorphic to $L_\rho$, and $\mathfrak{b}: L_A|_x \xrightarrow{\sim}A^n$ gives a basis at $x$. 

On the other hand, $\{\nabla\otimes \id_A+\omega\,|\, \omega\in\obj \sC(\sEnd_0(L_\rho); A)\}$ is the set of flat connections on the underlining $C^\infty$ vector bundle of $L_\rho \otimes_\cc A$ whose restriction to $L_\rho$ is $\nabla$ and which is equal to $\nabla\otimes \id_A$ at $x$.
Thus, via Riemann-Hilbert correspondence, each $\omega\in \obj \sC(\sEnd_0(L_\rho); A)$ gives a $A$-local system on $X_A$. Since $\omega|_x=0$, there is a base at $x$ canonically attached to the $A$-local system. Furthermore, the flat bundles given by $\omega, \omega'\in \obj \sC(\sEnd_0(L_\rho); A)$ are isomorphic to each other if and only if $\omega$ and $\omega'$ differ by a Gauge transformation, that is they are isomorphic to each other in $\sC(\sEnd_0(L_\rho); A)$. Therefore, $\homo(R, A)$ is functorially bijective to $\iso \sC(\sEnd_0(L_\rho); A)$.
\end{proof}

Restricting to $x$ gives an embedding of vector spaces $H^0(X, \enmo(L_\rho))\hookrightarrow \enmo(L_\rho)|_x$. We denote the quotient $(\enmo(L_\rho)|_x)/H^0(X, \enmo(L_\rho))$ by $V$, which can also be considered as an affine variety with a distinguished origin $0\in V$. 

Recall that $H^1(X, End(L_\rho))\cong Z^1(\pi_1(X), \mathfrak{gl}(n, \mathbb{C})_{\textrm{Ad}\rho})/B^1(\pi_1(X), \mathfrak{gl}(n, \mathbb{C})_{\textrm{Ad}\rho})$, and it's not hard to see that $B^1(\pi_1(X), \mathfrak{gl}(n, \mathbb{C})_{\textrm{Ad}\rho})$ is isomorphic to $V$ in a natural way. Thus, we have a short exact sequence,
$$
0\to V\to Z^1(\pi_1(X), \mathfrak{gl}(n, \mathbb{C})_{\textrm{Ad}\rho})\to H^1(X, End(L_\rho))\to 0.
$$
In the introduction, we have defined $C\subset Z^1(\pi_1(X), \mathfrak{gl}(n, \mathbb{C})_{\textrm{Ad}\rho})$ to be the kernel of the quadratic map $Z^1(\pi_1(X), \mathfrak{gl}(n, \mathbb{C})_{\textrm{Ad}\rho})\to H^2(X, \enmo(E))$ given by $\eta\mapsto \bar\eta\wedge\bar\eta$, where $\bar\eta$ is the image of $\eta$ in $H^1(X, \enmo(L_\rho))$ under the above homomorphism. Denote the image of $C$ under the projection $\pi: Z^1(\pi_1(X), \mathfrak{gl}(n, \mathbb{C})_{\textrm{Ad}\rho})\to H^1(X, End(L_\rho))$ by $\bar{C}$. Then $\bar C$ is the kernel of $H^1(X, End(L_\rho))\to H^2(X, End(L_\rho))$ defined by $\bar\eta\mapsto \bar\eta\wedge\bar\eta$.  Thus, $C\cong \bar{C}\times V$ in a non-canonical way. 

When $X$ is a compact K\"ahler manifold and $\rho$ is a semi-simple representation, the monodromy representation of the local system $\enmo(L_\rho)$ is also semi-simple. Corlette's construction in \cite{c} gives a harmonic metric on $\enmo(L_\rho)$. Using this harmonic metric, Simpson proved that $\sEnd(L_\rho)$ is formal. Then by analyzing the homomorphism $\sEnd_0(L_\rho)\hookrightarrow \sEnd(L_\rho)$, and applying Theorem \ref{deligne}, Goldman and Millson also showed the following. 

\begin{theorem}[Goldman-Millson, Simpson]\label{gm2}
Suppose $X$ is a compact K\"ahler manifold, and $\rho$ is a semi-simple representation. Denote the complete ring $(\sO_{C})_{(0)}$ by $R'$. Then the functor $A\mapsto \iso \sC(\sEnd_0(L_\rho); A)$ is represented by $R'$. 
\end{theorem}
An immediate consequence of Theorem \ref{gm1} and Theorem \ref{gm2} is the following.
\begin{cor}\label{gmcor}
The formal scheme of $\br(X, n)$ at $\rho$ is isomorphic to the formal scheme of $C$ at the origin. 
\end{cor}
For the rest of this section, we assume $X$ to be a compact K\"ahler manifold and assume $\rho$ to be a semi-simple representation. 

Given any $\eta\in Z^1(\pi_1(X), \mathfrak{gl}(n, \mathbb{C})_{\textrm{Ad}\rho})$, denote its image in $H^1(X, \enmo(E))$ by $\bar\eta$. Recall that the Aomoto complex of $\eta$ is defined to be 
$$(H^\bullet(X, L_\rho), \wedge\bar\eta)\stackrel{\textrm{def}}{=}H^0(X, L_\rho)\stackrel{\wedge \bar\eta}{\longrightarrow}H^1(X, L_\rho)\stackrel{\wedge \bar\eta}{\longrightarrow} H^2(X, L_\rho)\stackrel{\wedge \bar\eta}{\longrightarrow}\cdots.$$ 
The resonance variety is defined by 
\begin{equation}\label{resonancedef}
\sR^i_k(X, \rho)\stackrel{\textrm{def}}{=}\{\eta\in Z^1(\pi_1(X), \mathfrak{gl}(n, \mathbb{C})_{\textrm{Ad}\rho})\;| \;\dim H^i(H^\bullet(X, L_\rho), \wedge\bar\eta)\geq k\}.
\end{equation}
Notice that there is a tautological section $\zeta$ of $\sO_C\otimes_\cc H^1(X, \enmo(L_\rho))$, such that for any $\eta\in C$, $\zeta|_{\{\eta\}}=\bar\eta\in H^1(X, \enmo(L_\rho))$. Hence there is a universal Aomoto complex on $C$,
$$
(\sO_C\otimes_\cc H^\bullet(X, L_\rho), \wedge\zeta)=\sO_C\otimes_\cc  H^0(X, L_\rho)\xrightarrow{\wedge\zeta}\sO_C\otimes_\cc  H^1(X, L_\rho)\xrightarrow{\wedge\zeta}\cdots
$$
Now we give a precise definition of the resonance varieties as subschemes. 

\begin{defn}\label{resonance}
For any $i, r\in \zz_{\geq 0}$, $\sR^i_r(X, \rho)$ is defined to be the subscheme associated to the fitting ideal $I_{k_i-r+1}((\wedge\zeta)_{i-1}\oplus (\wedge\zeta)_i)$, where $k_i=\dim H^i(X, L_\rho)$, and $(\wedge\zeta)_{i-1}: H^{i-1}(X, L_\rho)\to H^i(X, L_\rho)$, $(\wedge\zeta)_{i}: H^{i}(X, L_\rho)\to H^{i+1}(X, L_\rho)$ are the differentials in the above complex. 
\end{defn}
Then one can easily check that the closed points of $\sR^i_r(X, \rho)$ satisfies (\ref{resonancedef}). 

Recall that $\bar{C}$ is the image of $C$ under the projection $\pi: Z^1(\pi_1(X), \mathfrak{gl}(n, \mathbb{C})_{\textrm{Ad}\rho})\to H^1(X, End(L_\rho))$. Let $\bar\sR^i_r(X, \rho)$ be the image of $\sR^i_r(X, \rho)$ under the same projection $\pi$. On $\bar C$, there is a similar tautological section $\bar\zeta$ of $\sO_{\bar C}\otimes_\cc H^1(X, \enmo(L_\rho))$, such that at any $\bar\eta\in \bar C$, $\bar\zeta|_{\bar\eta}=\bar\eta\in H^1(X, \enmo(L_\rho))$. Similarly, there is also a universal Aomoto complex on $\bar C$, 
$$
(\sO_{\bar C}\otimes_\cc H^\bullet(X, L_\rho), \wedge\bar\zeta)=\sO_{\bar C}\otimes_\cc  H^0(X, L_\rho)\xrightarrow{\wedge\bar\zeta}\sO_{\bar C}\otimes_\cc  H^1(X, L_\rho)\xrightarrow{\wedge\bar\zeta}\cdots
$$
Then, it is easy to see that $\bar\sR^i_r(X, \rho)$ is also defined by the determinantal ideals $I_{k_i-r+1}((\wedge\bar\zeta)_{i-1}\oplus (\wedge\bar\zeta)_i)$. Moreover, $\sR^i_r(X, \rho)=\pi^{-1}(\bar\sR^i_r(X, \rho))$. Recall that the isomorphism $C\cong \bar{C}\times V$ is not canonical. In other words, the isomorphism depends on a choice. However, by the previous argument, for any such choice the isomorphism maps $\sR^i_r(X, \rho)$ to $\bar\sR^i_r(X, \rho)\times V$.

Before proving Theorem \ref{thm1} and Theorem \ref{thm2}, we first show the formality result, which is essentially \cite[Lemma 2.2]{s1}.
\begin{prop}\label{formal}
The DGLA pair $(\sEnd(L_\rho), \odr(L_\rho))$ is formal. 
\end{prop}
\begin{proof}
The differentials $\nabla: \mathcal End^j(L_\rho)\to \mathcal End^{j+1}(L_\rho)$ and $\nabla: \Omega_{DR}^j(L_\rho)\to \Omega_{DR}^{j+1}(L_\rho)$ divides into $(1,0)$ and $(0,1)$ components $\nabla=\nabla^{(0,1)}+\nabla^{(1,0)}$. Denote the kernels of $\nabla^{(1,0)}$ by $\mathcal{KE}nd^j(L_\rho)$ and $\sK\Omega_{DR}^j(L_\rho)$ respectively. Then $(\mathcal{KE}nd^\bullet(L_\rho), \nabla^{(0, 1)})$ is a DGLA, and $(\sK\odr(L_\rho), \nabla^{(0, 1)})$ is a $(\mathcal{KE}nd^\bullet(L_\rho), \nabla^{(0, 1)})$-module. 

By definition, there is a natural injective map of DGLA pairs,
\begin{equation}\label{map1}
(\mathcal{KE}nd^\bullet(L_\rho), \sK\odr(L_\rho))\hookrightarrow (\sEnd(L_\rho), \odr(L_\rho)).
\end{equation}
The existence of harmonic metrics on $\enmo(L_\rho)$ and $L_\rho$ implies $\partial\bar\partial$-lemma holds. Hence, for every $j$, $\kn(\nabla^{(0, 1)})\cap \mathcal{KE}nd^j(L_\rho)\subset \kn(\nabla)$ and $\im(\nabla^{(0, 1)})\cap \mathcal{KE}nd^j(L_\rho)\subset \im(\nabla)$. Therefore, there is also a surjective map between DGLA pairs,
\begin{equation}\label{map2}
(\mathcal{KE}nd^\bullet(L_\rho), \sK\odr(L_\rho))\twoheadrightarrow (H^\bullet(\sEnd(L_\rho)), H^\bullet(\odr(L_\rho))).
\end{equation}
The existence of harmonic metrics on $\enmo(L_\rho)$ and $L_\rho$ also implies that the cohomology of every DGLA or DGLA module appearing in (\ref{map1}) or (\ref{map2}) is represented by harmonic forms. Thus, (\ref{map1}) and (\ref{map2}) are both homotopy equivalence. Therefore, $(\sEnd(L_\rho), \odr(L_\rho))$ is formal. 
\end{proof}

Now, we are ready to prove Theorem \ref{thm1} and Theorem \ref{thm2}. 
\begin{proof}[Proof of Theorem \ref{thm1}]
Since we only need to worry about the reduced part of the formal scheme, it is sufficient to prove the following statement.
\begin{statement}
Let $S$ be a noetherian complete local ring over $\cc$, which is also an integral domain. Suppose $h_1: \spec(S)\to \br(X, n)$ and $h_2: \spec(S)\to C$ map the closed point in $\spec(S)$ to $\rho$ and $0$ respectively. Passing to formal schemes, we have $\hat{h}_1: \widehat{\spec(S)}\to \br(X, n)_{(\rho)}$ and $\hat{h}_2: \widehat{\spec(S)}\to C_{(0)}$. Suppose the composition of $\hat{h}_1$ and the isomorphism $\br(X, n)_{(\rho)}\cong C_{(0)}$ is $\hat{h}_2$. Then the image of $h_1$ is contained in $\sV^i_r(X, n)$ if and only if the image of $h_2$ is contained in $\sR^i_r(X, \rho)$.
\end{statement}

Denote the pull back of the universal local system $\sL$ on $X\times \br(X, n)$ via $\id_X\times h_1: X\times \spec(S)\to X\times \br(X, n)$ by $L_S$. According to the isomorphism (\ref{char}), the image of $h_1$ is contained in $\sV^i_r(X, n)$ if and only if the rank of $H^i(X_S, L_S)$ is at least $r$ as an $S$-module. In other words, $\dim_K H^i(X_S, L_S)\otimes_S K\geq r$, where $K$ is the quotient field of $S$. 

On the other hand, $\pi: Z^1(\pi_1(X), \mathfrak{gl}(n, \mathbb{C})_{\textrm{Ad}\rho})\to H^1(X, End(L_\rho))$ induces a map $\pi_{(0)}: C_{(0)}\to \bar{C}_{(0)}$. Since $\sR^i_r(X, \rho)=\pi^{-1}(\bar\sR^i_r(X, \rho))$, the image of $h_2$ is contained in $\sR^i_r(X, \rho)$ if and only if the image of $\pi_{(0)}\circ \hat{h}_2$ is contained in $\bar\sR^i_r(X, \rho)_{(0)}$. Denote $\tilde{h}_2=\pi_{(0)}\circ \hat{h}_2: \spec(S)\to \bar{C}_{(0)}$. Then by the definition of $\bar\sR^i_r(X, \rho)$, the image of $\tilde{h}_2$ is contained in $\bar\sR^i_r(X, \rho)_{(0)}$ if and only if the global sections of $H^i((\sO_{\bar{C}}\otimes_\cc H^\bullet(X, L_\rho), \wedge\bar\zeta)\otimes_{\sO_{\bar{C}}} S)$ has rank at least $r$ as an $S$-module. Here we consider $S$ as an $\sO_{\bar{C}}$-module by composing $h_2: \spec(S)\to C$ with the projection $C\to \bar{C}$. Then the complex $(\sO_{\bar{C}}\otimes_\cc H^\bullet(X, L_\rho), \wedge\bar\zeta)\otimes_{\sO_{\bar{C}}} S$ is the pull-back of the universal Aomoto complex on $\bar{C}$ to $\spec(S)$. 

From now on, we denote $(\sO_{\bar{C}}\otimes_\cc H^\bullet(X, L_\rho), \wedge\bar\zeta)\otimes_{\sO_{\bar{C}}} S$ by $(S\otimes H^\bullet(X, L_\rho), \wedge\bar\zeta_S)$. And by abusing notations, we also write $H^i(S\otimes H^\bullet(X, L_\rho), \wedge\bar\zeta_S)$ for its global sections. Thus, to show the statement, it is enough to show the following isomorphism,
\begin{equation}\label{iso3}
H^i(X_S, L_S)\cong H^i(S\otimes H^\bullet(X, L_\rho), \wedge\bar\zeta_S).
\end{equation}

Denote the maximal ideal of $S$ by $m_S$, and denote the quotient $S/(m_S)^k$ by $S_k$. Then each $S_n$ is an Artinian local ring. $S$ is complete, and hence $S=\varprojlim\limits_{k}S_k$. Let $L_{S_k}=L_S\otimes_{S} S_k$ and $\bar\zeta_{S_k}=\bar\zeta_{S}\otimes_{S} S_k$. Since Mittag-Leffler condition is automatic for artinian modules, we have the following isomorphisms,
$$
H^i(X_S, L_S)\cong \varprojlim_{n}H^i(X_{S}, L_{S_k})
$$
and
$$
H^i(S\otimes H^\bullet(X, L_\rho), \wedge\bar\zeta_S)\cong \varprojlim_{n} H^i(S_k\otimes H^\bullet(X, L_\rho), \wedge\bar\zeta_{S_k}).
$$
Therefore, to show isomorphism (\ref{iso3}) is reduced to show the existence of diagram
\begin{equation}\label{diagram}
\begin{gathered}
\xymatrix{
H^i(X_{S}, L_{S_{k+1}})\ar[d]\ar[r]&H^i(S_{k+1}\otimes H^\bullet(X, L_\rho), \wedge\bar\zeta_{S_{k+1}})\ar[d]\\
H^i(X_{S}, L_{S_{k}})\ar[r]&H^i(S_k\otimes H^\bullet(X, L_\rho), \wedge\bar\zeta_{S_k})
}
\end{gathered}
\end{equation}
with both horizontal arrows being isomorphisms. 

In Theorem \ref{gm1}, we denoted the complete ring $(\sO_{\br(X, n)})_{(\rho)}$ by $R$. Since $S$ is complete, the map $h_1: \spec(S)\to \br(X, n)$ induces a ring homomorphism $h_1^*: R\to S$. Thus composing with $S\to S_k$, we have $(h_1^*)_k: R\to S_k$. Since $R$ represents the functor $A\mapsto \iso \sC(\sEnd_0(L_\rho); A)$, we can choose $\omega_k\in \obj \sC(\sEnd_0(L_\rho); S_k)$ such that the isomorphism class $\bar\omega_k$ of $\omega_k$ corresponds to the map $(h_1^*)_k: R\to S_k$. Furthermore, we can choose $\{\omega_k\}_{k\in\nn}$ in a compatible way, such that under the functor $\sC(\sEnd_0(L_\rho); S_{k+1})\to \sC(\sEnd_0(L_\rho); S_{k+1})$ induced by the natural projection $S_{k+1}\to S_k$, the image of $\omega_{k+1}$ is $\omega_k$. 

By the choice of $\omega_k$, one can easily check that $(\odr(L_\rho)\otimes S_k, \nabla\otimes \id_{S_k}+\omega_k)$ is the de Rham resolution of the local system $L_{S_k}$ on $X$. Here, $L_{S_k}$ is a locally constant sheaf on $X_{S_k}$, and hence a locally constant $S_k$-module over $X$. Therefore, 
$$
\hh^i_{\odr(L_\rho)}(\omega_k)\cong H^i(X, L_{S_k}).
$$
Notice when we compute $\hh^i_{\odr(L_\rho)}(\omega_k)$, it doesn't matter whether we consider $(\odr(L_\rho), \nabla)$ as a module over $\sEnd(L_\rho)$ or $\sEnd_0(L_\rho)$. Since $X\times \spec(S)\to X$ is affine, using Leray spectral sequence, we immediately have $H^i(X, L_{S_k})\cong H^i(X_{S}, L_{S_k})$. Therefore, 
\begin{equation}\label{iso4}
\hh^i_{\odr(L_\rho)}(\omega_k)\cong H^i(X_{S}, L_{S_k})
\end{equation}
We have chosen $\omega_k$ in the compatible way. Thus the isomorphism (\ref{iso4}) is functorial on $n$, that is we have the following commutative diagram,
\begin{equation}\label{diagram1}
\begin{gathered}
\xymatrix{
\hh^i_{\odr(L_\rho)}(\omega_{k+1})\ar[r]\ar[d]&H^i(X_{S}, L_{S_{k+1}})\ar[d]\\
\hh^i_{\odr(L_\rho)}(\omega_{k})\ar[r]&H^i(X_S, L_{S_{k}})
}
\end{gathered}
\end{equation}
where the horizontal arrows are isomorphism (\ref{iso4}), and the vertical arrows are induced from the natural projection $S_{k+1}\to S_k$. 

On the other hand, denote the complete ring $(\sO_{\bar{C}})_{(0)}$ by $R''$. Then $\hat{h}_2: \spec(S)\to \bar{C}_{(0)}$ induces $\hat{h}_2^{*}: R''\to S$. Composing with the natural projection $S\to S_k$, we have $(\hat{h}_2^{*})_k: R''\to S_k$. Recall that $(H^\bullet(\sEnd(L_\rho)), H^\bullet(\odr(L_\rho)))$ is the cohomology DGLA of the DGLA $(\sEnd(L_\rho), \odr(L_\rho))$. According to Definition \ref{cat}, the functor $A\mapsto \iso \sC(H^\bullet(\sEnd(L_\rho)); A)$ from the category of Artinian local rings over $\cc$ to the category of sets is obviously represented by $R''$. Choose $\sigma_k\in \obj\sC(H^\bullet(\sEnd(L_\rho)); S_k)$, whose isomorphism class $\bar\sigma_k$ corresponds to the ring homomorphism $(\hat{h}_2^{*})_k: R''\to S_k$. Then by the definition of the universal Aomoto complex $(\sO_{\bar C}\otimes_\cc H^\bullet(X, L_\rho), \wedge\bar\zeta)$, it is not hard to see that the complex $(H^\bullet(\odr(L_\rho))\otimes S_k, \sigma_k)$ is same as the global sections of the complex $H^i(S_k\otimes H^\bullet(X, L_\rho), \wedge\bar\zeta_{S_k})$. Thus, we have
\begin{equation}\label{iso5}
\hh^i_{H^\bullet(\odr(L_\rho))}(\sigma_k)\cong H^i(S_k\otimes H^\bullet(X, L_\rho), \wedge\bar\zeta_{S_k})
\end{equation}
Furthermore, we can assume we have chosen $\{\sigma_k\}_{k\in \nn}$ in a compatible way as before. Then we have the following commutative diagram
\begin{equation}\label{diagram2}
\begin{gathered}
\xymatrix{
\hh^i_{H^\bullet(\odr(L_\rho))}(\sigma_{k+1})\ar[r]\ar[d]&H^i(S_{k+1}\otimes H^\bullet(X, L_\rho), \wedge\bar\zeta_{S_{k+1}})\ar[d]\\
\hh^i_{H^\bullet(\odr(L_\rho))}(\sigma_k)\ar[r]&H^i(S_k\otimes H^\bullet(X, L_\rho), \wedge\bar\zeta_{S_k})
}
\end{gathered}
\end{equation}
where the horizontal arrows are isomorphism (\ref{iso5}), and the vertical arrows are induced by the natural projection $S_{k+1}\to S_k$. 

Now, thanks to Proposition \ref{homequ} and Proposition \ref{formal}, $\hh^i_{\odr(L_\rho)}(\omega_k)\cong \hh^i_{H^\bullet(\odr(L_\rho))}(\sigma_k)$, and this isomorphism is functorial. Hence we have a commutative diagram
\begin{equation}\label{diagram3}
\begin{gathered}
\xymatrix{
\hh^i_{\odr(L_\rho)}(\omega_{k+1})\ar[d]\ar[r] &\hh^i_{H^\bullet(\odr(L_\rho))}(\sigma_{k+1})\ar[d]\\
\hh^i_{\odr(L_\rho)}(\omega_{k})\ar[r] &\hh^i_{H^\bullet(\odr(L_\rho))}(\sigma_k)
}
\end{gathered}
\end{equation}
with horizontal arrows being isomorphisms. 

After all, (\ref{diagram}) follows from (\ref{diagram1}), (\ref{diagram2}) and (\ref{diagram3}). Thus we finished the proof of Theorem \ref{thm1}. 
\end{proof}

\begin{proof}[Proof of Theorem \ref{thm2}]
When $\dim H^i(X, L_\rho)=r$, it is clear that $\sR^i_r(X, \rho)$ is the intersection of $C$ and a linear subspace $H$ of $Z^1(\pi_1(X), \mathfrak{gl}(n, \cc)_{\ad \rho})$. $H$ contains all $\eta\in Z^1(\pi_1(X), \mathfrak{gl}(n, \cc)_{\ad \rho})$ satisfying both $\wedge\bar\eta: H^{i-1}(X, L_\rho)\to H^i(X, L_\rho)$ and $\wedge\bar\eta: H^i(X, L_\rho)\to H^{i+1}(X, L_\rho)$ are zero maps, where $\bar\eta$ is the image of $\eta$ in $H^1(X, \enmo(L_\rho))$. 

Let $A$ be an Artinian local ring over $\cc$. Let $g_1: \spec(A)\to \br(X, n)_{(\rho)}$, and let $g_2: \spec(A)\to C_{(0)}$ such that he composition of $g_1$ and $\br(X, n)_{(\rho)}\cong C_{(0)}$. It follows from \cite[Proposition 2.1]{w} (also, cf. Corollary 2.4 of \cite{w}) that the schematic image of $g_1$ is contained in $\sV^i_r(X, n)$ if and only if $H^i(X, L_A)$ is a free $A$-module of rank $r$. And it is obvious that the schematic image of $g_2$ is contained in $\sR^i_r(X, n)$ if and only if $H^i(A\otimes H^\bullet(X, L_\rho), \wedge\bar\zeta_A)$ is a free $A$-module of rank $r$. The proof of Theorem \ref{thm1} shows that 
$$H^i(X, L_A)\cong H^i(A\otimes H^\bullet(X, L_\rho), \wedge\bar\zeta_A). $$
Therefore, the schematic image of $g_1$ is contained in $\sV^i_r(X, n)$ if and only if the schematic image of $g_2$ is contained in $\sR^i_r(X, n)$. Thus we have proved Theorem \ref{thm2}. 
\end{proof}

\section{Cohomology jump loci in the moduli space of stable vector bundles}
The purpose of this section is to partially generalize the results in \cite{w} to any Hodge spectrum. The original method of \cite{w} doesn't work in this situation, because the wedge of a harmonic $(0,1)$-form and a harmonic $(p, q)$-form may not be harmonic anymore. Conversely, to conclude the whole result of \cite{w}, one has to argue that the description of the analytic germ of the moduli space of vector bundles according to Nadel \cite{n} and the one according to Goldman-Millson \cite{gm} coincide. I precise proof may require going through most of the calculation in \cite{w}. 

First, we review and extend the definition of cohomology jump loci in the moduli space of stable vector bundles, as in \cite{w}. Let $X$ be a compact K\"ahler manifold. Denote the moduli space of stable rank $r$ vector bundles with vanishing chern classes by $\sM(X, n)$. One can find a covering $\{U_\lambda\}_{\lambda\in \Lambda}$ of $\sM$, such that on each $U_\lambda$, there exist a Kuranishi family of vector bundle $\sE_\lambda$, or in other words, a complete family of vector bundles. Since we are only interested in the local properties of $\sM(X, r)$ and the cohomology jump loci, we will fix and only work on one of those $U_\lambda$. Denote $U_\lambda$, $\sE_\lambda$ by $U$, $\sE$ respectively, and the structure sheaf of $U$ by $\sO_U$. 

It is a direct consequence of the Grauert's direct image theorem and Mumford's construction that, after possibly passing to a smaller neighborhood, there exists a right bounded complex of free sheaves $F^\bullet$ on $U$ computing the cohomology of $\sE\otimes \Omega^p_{X\times U/U}$. More precisely, denote the projection from $X\otimes U$ to the first and second factor by $p_1$ and $p_2$ respectively. Then, for any $\sO_U$-module $G$ on $U$, 
\begin{equation}\label{compute}
\mathbf{R}p^q_{2*}(\sE\otimes\Omega^p_{X\times U/U}\otimes p_2^* G)\cong H^q(F^\bullet\otimes G).
\end{equation}

The locally defining ideal of $\sW^{pq}_r(X, n)$ is the determinantal ideal $I_{l_q-r+1}(d^{q-1}\oplus d^{q})$, where $l_q$ is the rank of $F^q$, and  $d^{q-1}: F^{q-1}\to F^q$, $d^q: F^q\to F^{q+1}$ are the differentials in $F^\bullet$. In fact, this is the definition in \cite{dp}, which may be different from the one in \cite{w}. However, they define the same closed points, and they are same along the locally closed loci $\sW^{pq}_r\setminus \sW^{pq}_{r+1}$. By putting $G=\cc_{u}$, for a closed point $u\in U$, (\ref{compute}) implies that set-theoretically $\sW^{pq}_r=\{E\in\sM(X, n)\;|\; \dim H^q(X, E\otimes_{\sO_X} \Omega^p_X)\geq r\}$. 

Let $E$ be a  rank $n$ stable vector bundle with vanishing chern classes. By abusing notation, we also denote its corresponding point in $\sM(X, n)$ by $E$. Denote the Dolbeault complex of $\enmo(E)$ by $(\sEndp(E), \bar\partial)$. According to \cite[9.4]{gm}, the deformation theory of $\sM(X, n)$ at $E$ is controlled by the DGLA $(\sEndp(E), \bar\partial)$. The following two theorems are indeed parallel to Theorem \ref{gm1} and Theorem \ref{gm2}. 

\begin{theorem}[Goldman-Millson]\label{gm3}
Denote the complete ring $(\sO_U)_{(E)}$ by $R$. Then $R$ represents the functor $A\mapsto \iso \sC(\sEndp(E); A)$.
\end{theorem}

\begin{theorem}\label{gm4}
Let $C$ be the kernel of the map $H^1(X, \enmo(E))\to H^2(X, \enmo(E))$ given by $\xi\mapsto \xi\wedge\xi$. Denote the complete ring $(\sO_C)_{(0)}$ by $R'$. Then $R'$ represents the functor $A\mapsto \iso \sC(\sEndp(E); A)$. 
\end{theorem}
Since $E$ is stable and of vanishing chern classes, $\enmo(E)$ is polystable and of vanishing chern classes too. It is proved by Uhlenbeck-Yau and Donaldson that there exists a Hermitian-Einstein metric on $\enmo(E)$, which has to be harmonic under the assumption of vanishing chern classes (see \cite{s1}). The existence of harmonic metric implies $\sEndp(E)$ is formal. Therefore, Theorem \ref{gm4} follows immediately from Theorem \ref{deligne} and the simple fact that $A\mapsto \iso \sC(H^\bullet(\sEndp(E)); A)$ is represented by $R'$. 

\begin{cor}
The formal scheme $\sM(X, n)_{(E)}$ is isomorphic to the formal scheme $C_{(0)}$. 
\end{cor}
Given any $\xi\in C$, the Aomoto complex with respect to $\xi$ is
$$
(H^\bullet(X, E\otimes \Omega^p_X), \wedge\xi)\stackrel{\textrm{def}}{=}H^0(X, E\otimes \Omega^p_X)\xrightarrow{\wedge\xi}H^1(X, E\otimes \Omega^p_X)\xrightarrow{\wedge\xi}H^2(X, E\otimes \Omega^p_X)\xrightarrow{\wedge\xi}\cdots
$$
Similarly to Section 3, there is a tautological section $\chi$ of $\sO_C\otimes_\cc H^1(X, \enmo(E))$ such that the complex
$$
(\sO_C\otimes H^\bullet(X, E\otimes \Omega^p_X), \wedge\chi)\stackrel{\textrm{def}}{=}\sO_C\otimes H^0(X, E\otimes \Omega^p_X)\xrightarrow{\wedge\chi}\sO_C\otimes H^1(X, E\otimes \Omega^p_X)\xrightarrow{\wedge\chi}\cdots
$$
is the universal Aomoto complex on $C$. We define the resonance variety $\sS^{pq}_r(X, E)$ to be the analytic subspace associated to the determinantal ideal $I_{l-r+1}((\wedge\chi)_{q-1}\oplus (\wedge\chi)_q)$, where $l=\dim H^q(X, E\otimes \Omega^p_X)$ and $(\wedge\chi)_{q-1}: H^{q-1}(X, E\otimes \Omega^p_X)\to H^q(X, E\otimes \Omega^p_X)$, $(\wedge\chi)_{q}: H^{q}(X, E\otimes \Omega^p_X)\to H^{i+1}(X, E\otimes \Omega^p_X)$ are the differentials in the above complex. Then set-theoretically, we have
$$\sS^{pq}_r(X, E)=\{\xi\in C\;|\; \dim_\cc H^q(H^\bullet(X, E\otimes \Omega^p_X), \wedge\xi)\geq r\}. $$

The parallel theorems of Theorem \ref{thm1} and Theorem \ref{thm2} are the following. 
\begin{theorem}\label{thm3}
The isomorphism between the formal schemes
$$
\sM(X, n)_{(E)}\cong C_{(0)}
$$
induces an isomorphism between the reduced formal schemes
$$
\sW^{pq}_r(X, n)_{(E), \red}\cong \sS^{pq}_r(X, E)_{(0), \red}
$$
Furthermore, when $\dim H^q(X, E\otimes \Omega^p_X)=r$, there is an isomorphism before taking the reduced induced structures, that is 
$$
\sW^{pq}_r(X, n)_{(E)}\cong \sS^{pq}_r(X, E)_{(0)}.
$$
\end{theorem}
The proof is also parallel to the proof of Theorem \ref{thm1} and Theorem \ref{thm2}, and even easier. So we just sketch the proof. 
\begin{proof}[Sketch of proof]
Let $S$ be a noetherian complete local ring over $\cc$, which is also an integral domain. Consider maps $g_1: \spec(S)\to \sM(X, n)$ and $g_2: \spec(S)\to C$ mapping the closed point  of $\spec(S)$ to $E$ and $0$ respectively. Passing to formal schemes, they become $\hat{g}_1: \widehat{\spec(S)}\to \sM(X, n)_{(E)}$ and $\hat{g}_2: \widehat{\spec(S)}\to C_{(0)}$ respectively. Suppose the composition of $\hat{g}_1$ and the isomorphism $\sM(X, n)_{(E)}\cong C_{(0)}$ is $\hat{g}_2$. Let $X_S=X\times_\cc \spec(S)$ and let $E_S$ be the pull back of the local Kuranishi family $\sE$ by $\id_X\times g_1: X\times \spec(S)\to X\times \sM(X, n). $
Similarly denote by $(H^\bullet(X, E\otimes \Omega^p_X)\otimes S, \wedge \chi_S)$ the pull back of the universal Aomoto complex $(H^\bullet(X, E\otimes \Omega^p_X)\otimes \sO_C, \wedge\chi)$ by $\id_X\times g_2: X\times \spec(S)\to X\times C$.

To prove the first part of the theorem, we assume $S$ is an integral domain. Then it is enough to show that the global sections of $H^q(X_S, E_S\otimes_{\sO_X}\Omega^p_X)$ is isomorphic to $H^q(H^\bullet(X, E\otimes \Omega^p_X)\otimes S, \wedge\chi_S)$. 
Let $S_k=S/(m_S)^k$, $E_{S_k}=E_S\otimes_S S_k$ and $\chi_{S_k}=\chi_S\otimes_S S_k$. Just as the proof of isomorphisms (\ref{iso4}), (\ref{iso5}) and (\ref{diagram3}), we can prove the following isomorphisms (\ref{iso21}), (\ref{iso22}) and (\ref{iso23}). 
\begin{equation}\label{iso21}
H^q(X_{S}, E_{S_k}\otimes_{\sO_X} \Omega^p_X)\cong \hh^q_{\odol(E)}(\omega_k)
\end{equation}
where $\odol(E)$ is the Dolbeault resolution of $E\otimes \Omega^p_X$, and $\omega_k\in \obj\sC(\sEndp(E); S_k)$ whose isomorphism class corresponds to the composition of $g_1^*: R\to S$ and the projection $S\to S_k$ in the sense of Theorem \ref{gm3}. 
\begin{equation}\label{iso22}
H^q(S_k\otimes H^\bullet(X, E\otimes \Omega^p_X), \wedge\chi_{S_k})\cong \hh^q_{H^\bullet(\odol(E))}(\sigma_k)
\end{equation}
where $\sigma_k\in \obj\sC(H^\bullet(\sEndp(E)); S_k)$ whose isomorphism class corresponds to the composition of $g_2^*: R'\to S$ and the projection $S\to S_k$. 
\begin{equation}\label{iso23}
\hh^q_{\odol(E)}(\omega_k)\cong \hh^q_{H^\bullet(\odol(E))}(\sigma_k)
\end{equation}
Therefore, 
$$
H^q(X_{S}, E_{S_k}\otimes_{\sO_X} \Omega^p_X)\cong H^q(S_k\otimes H^\bullet(X, E\otimes \Omega^p_X), \wedge\chi_{S_k})
$$
and this isomorphism is functorial on $S_k$. Since all the Mittag-Leffler condition is obviously satisfied, by taking inverse limit, we have
\begin{equation}\label{iso24}
H^q(X_{S}, E_{S}\otimes_{\sO_X} \Omega^p_X)\cong H^q(S\otimes H^\bullet(X, E\otimes \Omega^p_X), \wedge\chi_{S}).
\end{equation}
Thus, we proved the first part of the theorem.

Comparing to the proof of Theorem \ref{thm1}, isomorphism (\ref{iso24}) being true for any Artinian local ring $S$ implies the second part of the theorem. 
\end{proof}

\section{Applications and conjectures}
In fact, my motivation to study the cohomology jump loci is to study the fundamental groups of compact K\"ahler manifolds and their representations. Theorem \ref{thm1} and Theorem \ref{thm2} give new restrictions for a group being isomorphic to the fundamental group of some compact K\"ahler manifolds. 

Let $\Gamma$ be a finitely presented group. The set of representations $\homo(\Gamma, GL(n, \cc))$ has naturally a scheme structure, which we denote by $\br(\Gamma, n)$. Each representation $\rho\in \br(\Gamma, n)$ gives $\cc^n$ a $\Gamma$-module structure, which we denote by $(\cc^n)_\rho$. The cohomology jump loci $\sV^i_r(\Gamma, n)\stackrel{\textrm{def}}{=}\{\rho\in\br(\Gamma, n)\;|\; \dim H^i(\Gamma, (\cc^n)_\rho)\geq r\}$ are Zariski closed subsets of $\br(\Gamma, n)$. In fact, they can be defined more precisely as subschemes, because they can be identified with the cohomology jump loci for the classifying space of $\Gamma$. 

If $\Gamma\cong \pi_1(X)$ for some compact K\"ahler manifold $X$, then $\br(\Gamma, n)\cong \br(X, n)$. And furthermore, the isomorphism induces isomorphism of subschemes $\sV^1_r(\Gamma, n)\cong \sV^1_r(X, n)$. Therefore, Corollary \ref{consequence} gives rise to conditions on $\sV^1_r(\Gamma, n)$.
\begin{cor}
Let $\rho\in \br(\Gamma, n)$ be a semi-simple representation, and let $r=\dim H^1(\Gamma, (\cc^n)_\rho)$. Then $\sV^1_r(\Gamma, n)$ has quadratic singularity at $\rho$. 
\end{cor}

It would be interesting to have some example of finitely presented group which satisfies the condition of Goldman-Millson, but fails this condition. We also want to mention some potential connection between our results here and the absolute constructible sets defined by Simpson. 

Now, we assume $X$ to be a complex smooth projective variety. Denote the moduli space of $Gl(n, \cc)$ representations of $\pi_1(X)$, the moduli space of flat bundles on $X$ of rank $n$ and the moduli space of semistable Higgs bundles on $X$ of rank $n$ with trivial chern classes by $\mb(X, n)$, $\mdr(X, n)$ and $\mdol(X, n)$ respectively. In \cite{s3}, it is proved that there is a natural analytic isomorphism $\phi: \mb(X, n)\to\mdr(X, n)$, and a natural homeomorphism $\psi: \mb(X, n)\to\mdol(X, n)$. In \cite{s3}, inspired by Deligne's definition of absolute Hodge cycles, Simpson defined absolute constructible sets in $\mb(X, n)$. The cohomology jump loci in each moduli space are identified via $\phi$ and $\psi$, and they are the main examples of closed absolute constructible sets. 

In the case of $n=1$, Simpson has proved that the closed absolute constructible sets are unions of torsion translates of subtori. Theorem \ref{thm1} and Theorem \ref{thm2} may give a hint of some possible local conditions of absolute constructible sets, which would partially generalize Simpson's result. In fact, when $n=1$, Theorem \ref{thm2} together with an explicit description of the isomorphism in Corollary \ref{gmcor} by the exponential map shows that the cohomology jump loci has to be union of translates of subtori (see \cite{w}). Instead of Theorem \ref{thm2}, Theorem \ref{thm1} would also imply the same conclusion. In fact, if an irreducible subvariety $V$ of a complex torus $T$ satisfies that near every point it is the image of a cone under the exponential map, then $V$ must be a translate of a subtorus\footnotemark\ . 

\footnotetext{This can be proved easily by looking at a smooth point of $V$. In fact, Arapura has proved a much stronger statement that if this is true for one point of $V$, then $V$ has to be a translate of a subtorus. See \cite{a1} and \cite{a2}. }

So we propose the analogy of being unions of translates of subtori as follows. Similar to Theorem \ref{thm3}, we can prove the following. For an irreducible representation $\rho\in \mb(X, n)$, there is a quadratic cone $C$ in $H^1(X, \enmo(L_\rho))$ such that there is a canonical isomorphism $\mb(X, n)_{(\rho)}\cong C_{(0)}$. Our Theorem \ref{thm2} and  \cite[Proposition 10.5]{s3} imply that under this isomorphism, each cohomology jump locus corresponds to a cone in $C_{(0)}$. In general, suppose $S\subset \mb(X, n)$ is a closed absolute constructible subset and it contains a point corresponding to an irreducible representation $\rho$. Then under the above isomorphism, we conjecture that $S_{(\rho)}$ must be a cone in $C_{(0)}$.


\begin{thebibliography}{9}
\bibitem[A1]{a1}D. Arapura, \textsl{Higgs line bundles, Green-Lazarsfeld sets, and maps of K\"ahler manifolds to curves.} Bull. Amer. Math. Soc. 26 (1992), no. 2, 310-314.
\bibitem[A2]{a2}D. Arapura, \textsl{Geometry of cohomology support loci for local systems. I.} J. Algebraic Geom. 6 (1997), no. 3, 563-597.
\bibitem[Ar]{ar}M. Artin, \textsl{On the solutions of analytic equations.} Invent. Math. 5 (1968), 277-291. 
\bibitem[BW]{bw}N. Budur, B. Wang, \textsl{Cohomology jump loci of quasi-projective varieties.} arXiv: 1211.3766.
\bibitem[C]{c}K. Corlette, \textsl{Flat $G$-bundles with canonical metrics.} J. Differential Geom. 28 (1988), no. 3, 361-382.
\bibitem[DPS]{dps}A. Dimca, S. Papadima, A. Suciu, \textsl{Topology and geometry of cohomology jump loci.} Duke Math. J. 148 (2009), no. 3, 405-457.
\bibitem[DP]{dp}A. Dimca, S. Papadima, \textsl{Nonabelian cohomology jump loci from an analytic viewpoint.} arXiv:1206.3773.
\bibitem[E]{eisenbud}D. Eisenbud, \textsl{Commutative algebra. With a view toward algebraic geometry.} Graduate Texts in Mathematics, 150. Springer-Verlag, New York, 1995. 
\bibitem[GL1]{gl1}M. Green, R. Lazarsfeld, \textsl{Deformation theory, generic vanishing theorems, and some conjectures of Enriques, Catanese and Beauville.} Invent. Math. 90 (1987), no. 2, 389-407.
\bibitem[GL2]{gl2}M. Green, R. Lazarsfeld, \textsl{Higher obstructions to deforming cohomology groups of line bundles.} J. Amer. Math. Soc. 4 (1991), no. 1, 87-103.
\bibitem[GM]{gm}W. Goldman, J. Millson, \textsl{Deformations of flat bundles over K\"ahler manifolds.} Inst. Hautes \'Etudes Sci. Publ. Math. No. 67 (1988), 43-96. 
\bibitem[LM]{lm}A. Lubotzky, A. R. Magid, \textsl{Varieties of representations of finitely generated groups.} Mem. Amer. Math. Soc. 58 (1985), no. 336.
\bibitem[Na]{n}A. Nadel, \textsl{Singularities and Kodaira dimension of the moduli space of flat Hermitian-Yang-Mills connections.} Compositio Math. 67 (1988), no. 2, 121-128.
\bibitem[S1]{s1}C. Simpson, \textsl{Higgs bundles and local systems.} Inst. Hautes \'Etudes Sci. Publ. Math. No. 75 (1992), 5-95.
\bibitem[S2]{s2}C. Simpson, \textsl{Subspaces of moduli spaces of rank one local systems.} Ann. Sci. \'Ecole Norm. Sup. (4) 26 (1993), no. 3, 361-401. 
\bibitem[S3]{s3}C. Simpson, \textsl{Moduli of representations of the fundamental group of a smooth varprojlim variety. II.} Inst. Hautes \'Etudes Sci. Publ. Math. No. 80 (1994), 5-79
\bibitem[SS]{ss} M. Schlessinger, J. Stasheff, \textsl{Deformation theory and rational homotopy type} arXiv:1211.1647
\bibitem[W]{w}B. Wang, \textsl{Cohomology jump loci in the moduli spaces of vector bundles.} arXiv:1210.1487
\end{thebibliography}
\end{document}